\documentclass[12pt]{article}
\usepackage{amsmath}
\usepackage{amsfonts}

\newtheorem{theorem}{Theorem}

\newtheorem{corollary}{Corollary}

\newtheorem{lemma}[theorem]{Lemma}

\newtheorem{remark}{Remark}

\newenvironment{proof}[1][Proof]{\noindent\textbf{#1.} }{\ \rule{0.5em}{0.5em}}
\begin{document}

\title{Exponential inequalities for Mann's iterative scheme with functional
random errors}
\author{Bahia BARACHE$^{1}$, Idir ARAB$^{2},$ Abdelnasser\ DAHMANI$^{3}$ \\
$^{1}$Laboratoire de Math\'{e}matiques Appliqu\'{e}es\\
Facult\'{e} des Sciences Exactes, Universit\'{e} A.MIRA B\'{e}jaia,
Algerie\\
$^{2}$CMUC, Department of Mathematics, University of Coimbra, Portugal\\
$^{3}$Centre Universitaire de Tamanrasset\\
e-mail:$^{1}$Bahiabarache.maths@gmail.com; $^{2}$idir@mat.uc.pt\\
$^{3}$a\_dahmany@yahoo.fr}
\date{ }
\maketitle

\begin{abstract}
In this paper, we deal with an iteration method for approximating a fixed
point of a contraction mapping using the Mann's algorithm under functional
random errors. We first show its almost complete convergence to the fixed
point by mean of an exponential inequality and then we specify the induced
rate of convergence. We finally build a confidence set for the fixed point.%
\newline

\newline

\end{abstract}

\section{Introduction}

The main objective of studies in the fixed point theory is to find solutions
for the following equation, which is commonly known as fixed point equation:%
\begin{equation}
F\left( x\right) =x  \label{pointfixe}
\end{equation}%
where $F$ is a self-map of an ambient space $X$ and $x\in X$.

The most well-known result in fixed point theory is Banach's contraction
mapping principle; it guarantees that a contraction mapping of a complete
metric space to itself has a unique fixed point which may be obtained as the
limit of an iteration scheme defined by repeated images under the mapping of
an arbitrary starting point in the space. As such, it is a constructive
fixed point theorem and hence, may be implemented for the numerical
computation of the fixed point.

To solve equations given by (\ref{pointfixe}), two types of methods are
normally used: direct methods and iterative methods. Due to various reasons,
direct methods can be impractical or fail in solving equations (\ref%
{pointfixe}) because it leads to the inversion of a certain function, thing
that is not easy to do and thus, iterative methods become a viable
alternative. For this reason, the iterative approximation of fixed points
has become one of the major and basic tools in the theory of equations.

Mann \cite{Man} introduced an iterative scheme and employed it to
approximate the solution of a fixed point problem defined by nonexpansive
mapping where Picard's iterative scheme fails to converge. Later, Ishikawa
\cite{Ish} introduced an iterative method to obtain the convergence of a
Lipschitzian pseudo-contractive operator when Mann's iterative scheme is not
applicable. Many authors studied the convergence theorems and stability
problems in Banach spaces and metric spaces (see; e.g. \cite%
{Ber,Ceg,Cha,Kan,Osi2,Pan,Sha,HXu}) using the Mann's iteration scheme or the
Ishikawa's iteration scheme in deterministic frame. Some theoretical results
on Mann-Ishikawa algorithm with errors can be found in various literatures,
(e.g. see \cite{Aga,Cha2,Hua,Kaz,Kim,Liu,Liu2,Liu3,ZLiu2,Osi,YXu,YXu2}).

In the last twenty years, many papers have been published on the random
fixed point theory. The study of random fixed point theory is playing an
increasing role in mathematics and engineering sciences. Recently, it
received considerable attention due to enormous applications in many
important areas such as nonlinear analysis, probability theory and for the
study of random equations arising in various engineering sciences.

Choudhury \cite{Cho,Cho1} has suggested and analyzed random Mann's iterative
sequence in separable Hilbert spaces for finding random solutions and random
fixed points for some kind of random equations and random operators. Okeke
and Kim \cite{Oke} introduced the random Picard-Mann hybrid iterative
process. They have established the strong convergence theorems and summable
almost $T$-stability of the random Picard-Mann hybrid iterative process and
the random Mann-type iterative process generated by a generalized class of
random operators in separable Banach spaces.

Chugh et al. \cite{Chu}\ studied the strong convergence and stability of a
new two-step random iterative scheme with errors for accretive Lipschitzian
mapping in real Banach spaces. In \cite{Cho0}, Cho et al. has built a random
Ishikawa's iterative sequence with errors for random strongly
pseudo-contractive operator in separable Banach spaces and proved that under
suitable conditions, this random iterative sequence with errors converges to
a random fixed point of the operator.

Saluja et al. \cite{Sal} proved that if a random Mann's iteration scheme
defined by two random operators is convergent under some contractive
inequality, the limit point is a common fixed point of each of two random
operators in Banach space.

In \cite{Ar}, a random fixed point theorem was obtained for the sum of a
weakly-strongly continuous random operator and a nonexpansive random
operator which contains as a special Krasnoselskii type of Edmund and
O'Regan via the method of measurable selectors. We note some recent works on
random fixed points in \cite{Abe,Agr,Beg,Chu0,Chu,Hus,Oke}.

In this paper, we deal with iteration methods for approximating a fixed
point of the function using the Mann's algorithm with functional random
errors. We first show its complete convergence to the fixed point by mean of
an exponential inequality. This inequality will allow us to specify a
convergence rate and the possibility of building a confidence set for the
present fixed point.

\subsection{Some fixed point algorithms}

Let $X$ be a normed linear space and $F:X\rightarrow X$ a given operator.
Let $x_{0}\in X$ be arbitrary. The sequence $(x_{n})_{n}\subset X$ defined by%
\begin{equation}
x_{n+1}=F\left( x_{n}\right)  \label{Picard}
\end{equation}%
is called the Picard's iteration \cite{Pic}.

The sequence $(x_{n})_{n}\subset X$ defined by
\begin{equation}
x_{n+1}=\left( 1-a_{n}\right) x_{n}+a_{n}F\left( x_{n}\right) ,n\in \mathbb{N%
}^{\ast }  \label{Mann}
\end{equation}%
where $(a_{n})_{n}$ is a real sequence of positive numbers satisfying the
following conditions%
\begin{eqnarray*}
1.\text{ } &&a_{0}=1 \\
2.\text{ } &&0\leq a_{n}<1,\forall \ n\in \mathbb{N}^{\ast } \\
3.\text{ } &&\sum_{n}a_{n}=+\infty
\end{eqnarray*}%
is called the Mann's iteration or Mann's iterative scheme \cite{Man}.

The sequence $(x_{n})_{n}\subset X$ defined by%
\begin{eqnarray}
x_{n+1} &=&\left( 1-a_{n}\right) x_{n}+a_{n}F\left( y_{n}\right) ,n\in
\mathbb{N}^{\ast }  \label{Ishikawa} \\
y_{n} &=&\left( 1-b_{n}\right) x_{n}+b_{n}F\left( x_{n}\right) ,n\in \mathbb{%
N}^{\ast }  \notag
\end{eqnarray}%
where $(a_{n})_{n}$ and $(b_{n})_{n}$ are real sequences of positive numbers
satisfying the conditions%
\begin{eqnarray*}
1.\text{ } &&0\leq a_{n},b_{n}<1\text{ for all }n \\
2.\text{ } &&\lim_{n\rightarrow +\infty }b_{n}=0 \\
3.\text{ } &&\sum_{n}a_{n}b_{n}=+\infty
\end{eqnarray*}

and $x_{0}\in X$ is arbitrary. This procedure is called the Ishikawa's
iteration or Ishikawa's iterative procedure \cite{Ish}.

The sequence $(x_{n})_{n}\subset X$ defined by

\begin{equation*}
x_{n+1}=\frac{1}{2}\left( F\left( x_{n}\right) +x_{n}\right)
\end{equation*}%
is called the Krasnoselskii's iteration \cite{Kra}.

\begin{remark}
For $a_{n}=\frac{1}{2}$, the iteration (\ref{Mann}) reduces to the so-called
Krasnoselskii's iteration while for $a_{n}=1$ we obtain the Picard's
iteration (\ref{Picard}), or the method of successive approximations, as it
is commonly known. Obviously, for $b_{n}=0$ the Ishikawa's iteration (\ref%
{Ishikawa}) reduces to (\ref{Mann}).
\end{remark}

\section{Preliminaries}

Let $\left( \Omega ,\mathcal{F},\mathbb{P}\right) $ be a probability space
and $\mathbb{B}$ a real separable Banach space. Let $\left( \mathbb{B},%
\mathfrak{B}\right) $ be a measurable space, where $\mathfrak{B}$ denotes
the $\sigma $-algebra of all Borel subsets generated by all open subsets in $%
\mathbb{B}$, and $F:\mathbb{B\rightarrow B}$ a contraction mapping.
\begin{equation*}
\forall \ x,y\in \mathbb{B},\left\Vert F\left( x\right) -F\left( y\right)
\right\Vert \leq c\left\Vert x-y\right\Vert ,c\in \left[ 0,1\right) .
\end{equation*}%
Under this condition, the Banach's fixed point theorem states that $F$ has a
unique fixed point $x^{\ast }$.

Let $\left( x_{n}\right) _{n}$ be a sequence obtained by a certain fixed
point iteration procedure that ensures its convergence to a fixed point $%
x^{\ast }$ of $F$. Specifically for the Mann's algorithm, when calculating $%
\left( x_{n}\right) _{n}$, we usually follow these steps:

\begin{enumerate}
\item We choose the initial approximation $x_{0}\in \mathbb{B}$;

\item We compute $x_{1}=\left( 1-a_{0}\right) x_{0}+a_{0}F\left(
x_{0}\right) $ but, due to various errors that occur during the computations
(rounding errors, numerical approximations of functions, derivatives or
integrals, etc.), we do not get the exact value of $x_{1}$, but a different
one, say $y_{1}$, which is however close enough to $x_{1}$, i.e., $%
y_{1}-x_{1}=\xi _{1}$.

\item Consequently, when computing $x_{2}=\left( 1-a_{1}\right)
x_{1}+a_{1}F\left( x_{1}\right) ,$ we will actually compute $x_{2}$ as $%
x_{2}=\left( 1-a_{1}\right) y_{1}+a_{1}F\left( y_{1}\right) $ and so,
instead of the theoretical value $x_{2}$, we will obtain in fact another
value, say $y_{2}$, again close enough to $x_{2}$, i.e., $y_{2}-x_{2}=\xi
_{2},$ $\cdots ,$ and so on.
\end{enumerate}

In this way, instead of the theoretical sequence $\left( x_{n}\right) _{n}$
defined by the given iterative method, we will practically obtain an
approximate sequence $\left( y_{n}\right) _{n}$. We shall consider the given
fixed point iteration method to be numerically stable if and only if, for $%
y_{n}$ close enough (in some sense) to $x_{n}$ at each stage, the
approximate sequence $\left( y_{n}\right) _{n}$ still converges to the fixed
point of $F.$ That is to say,%
\begin{equation*}
x_{n+1}=\left( 1-a_{n}\right) x_{n}+a_{n}F\left( x_{n}\right) +\xi _{n}.
\end{equation*}

Unfortunately, the definitions of Liu \cite{Liu}, which depend on the
convergence of the error terms, is against the randomness of errors. Hence,
we need a new definition as follows%
\begin{equation*}
x_{n+1}=\left( 1-a_{n}\right) x_{n}+a_{n}F\left( x_{n}\right) +b_{n}\xi _{n},
\end{equation*}%
with $\left( \xi _{n}\right) _{n}$ a sequence of independent functional
random variables denoting noise which is defined on $\left( \Omega ,\mathcal{%
F},\mathbb{P}\right) $ with values into Banach spaces $\mathbb{B}$.
Moreover, assume that $\left( \xi _{n}\right) _{n}$ is zero mean and $%
\underset{n}{\sup }\ \mathbb{E}\left\Vert \xi _{n}\right\Vert <\infty $.

In this paper, we use the following stochastic Mann's algorithm%
\begin{equation}
x_{n+1}=\left( 1-a_{n}\right) x_{n}+a_{n}F\left( x_{n}\right) +b_{n}\xi _{n},
\label{smi}
\end{equation}%
satisfying%
\begin{equation*}
\sum_{n=1}^{\infty }a_{n}=\infty
\begin{array}{ccc}
& \text{and} &
\end{array}%
\sum_{n=1}^{\infty }b_{n}<\infty .
\end{equation*}%
(The condition $\sum\limits_{n=1}^{+\infty }a_{n}=+\infty $ is sometimes
replaced by $\sum\limits_{n=1}^{\infty }a_{n}\left( 1-a_{n}\right) =+\infty $%
).

Without loss of generality, we take
\begin{equation*}
a_{n}=\frac{a}{n}\text{ and }b_{n}=\frac{a}{n^{2}}\ 
\end{equation*}

In this case, the stochastic Mann's algorithm (\ref{smi}) takes the form%
\begin{equation}
x_{n+1}=\left( 1-\frac{a}{n}\right) x_{n}+\frac{a}{n}\left[ F\left(
x_{n}\right) +\frac{1}{n}\xi _{n}\right] .  \label{mannstocha}
\end{equation}

\begin{lemma}
By using the formula of the algorithm (\ref{mannstocha}), one obtains for $%
\left\Vert x_{1}-x^{\ast }\right\Vert \leq N$
\begin{equation}
\left\Vert x_{n+1}-x^{\ast }\right\Vert \leq N\prod\limits_{i=1}^{n}\left(
1-\frac{a\left( 1-c\right) }{i}\right) +\sum\limits_{i=1}^{n}\frac{a}{i^{2}}%
\prod\limits_{j=i+1}^{n}\left( 1-\frac{a\left( 1-c\right) }{j}\right)
\left\Vert \xi _{i}\right\Vert .  \label{formule}
\end{equation}
\end{lemma}

\begin{proof}
By adding and subtracting $x^{\ast }$ and using that $F(x^{\ast })=x^{\ast }$%
, we obtain%
\begin{equation*}
x_{n+1}-x^{\ast }=\left( 1-\frac{a}{n}\right) \left( x_{n}-x^{\ast }\right) +%
\frac{a}{n}\left[ F\left( x_{n}\right) -F\left( x^{\ast }\right) +\frac{1}{n}%
\xi _{n}\right] .
\end{equation*}

Using the last formula and the contraction of $F$, we get%
\begin{eqnarray*}
\left\Vert x_{n+1}-x^{\ast }\right\Vert &\leq &\left( 1-\frac{a\left(
1-c\right) }{n}\right) \left\Vert x_{n}-x^{\ast }\right\Vert +\frac{a}{n^{2}}%
\left\Vert \xi _{n}\right\Vert \\
&\leq &\left\Vert x_{1}-x^{\ast }\right\Vert \prod\limits_{i=1}^{n}\left( 1-%
\frac{a\left( 1-c\right) }{i}\right) +\sum\limits_{i=1}^{n}\frac{a}{i^{2}}%
\prod\limits_{j=i+1}^{n}\left( 1-\frac{a\left( 1-c\right) }{j}\right)
\left\Vert \xi _{i}\right\Vert \\
&\leq &N\prod\limits_{i=1}^{n}\left( 1-\frac{a\left( 1-c\right) }{i}\right)
+\sum\limits_{i=1}^{n}\frac{a}{i^{2}}\prod\limits_{j=i+1}^{n}\left( 1-%
\frac{a\left( 1-c\right) }{j}\right) \left\Vert \xi _{i}\right\Vert ,
\end{eqnarray*}%
as required.
\end{proof}

\begin{lemma}
For all positive constant $a$ such that $0<a<1,$ we have the following
inequality%
\begin{equation}
\prod\limits_{j=i+1}^{n}\left( 1-\frac{a\left( 1-c\right) }{j}\right) \leq
\left( \frac{i+1}{n+1}\right) ^{a\left( 1-c\right) }.  \label{produit}
\end{equation}
\end{lemma}

\begin{proof}
We have,

\begin{equation*}
\prod\limits_{j=i+1}^{n}\left( 1-\frac{a\left( 1-c\right) }{j}\right) \leq
\exp \left( -a\left( 1-c\right) \sum_{j=i+1}^{n}\frac{1}{j}\right) \leq
\left( \frac{i+1}{n+1}\right) ^{a\left( 1-c\right) },
\end{equation*}%
which is what had to be shown.
\end{proof}

\section{Main results}

\subsection{Exponential inequalities}

In this subsection, we establish an exponential inequality of
Bernstein-Frechet type for the stochastic Mann's scheme.

\begin{theorem}
\label{Theoreme}For all $\varepsilon >0,$ if for some constants $\sigma $
and $L>0$ the inequalities%
\begin{equation}
\mathbb{E}\left\Vert \xi _{i}\right\Vert ^{m}\leq \frac{m!}{2}\sigma
^{2}L^{m-2}  \label{Cramer}
\end{equation}%
are fulfilled, and if we denote by%
\begin{equation*}
S_{1}=\sum_{i=1}^{\infty }\frac{\left( i+1\right) ^{a\left( 1-c\right) }}{%
i^{2}}\text{ and }S_{2}=4a^{2}\sigma ^{2}\sum_{i=1}^{\infty }\frac{\left(
i+1\right) ^{2a\left( 1-c\right) }}{i^{4}}
\end{equation*}%
then%
\begin{equation}
\mathbb{P}\left\{ \left\Vert x_{n+1}-x^{\ast }\right\Vert >\varepsilon
\right\} \leq K_{1}\exp \left( -K_{2}n^{2a\left( 1-c\right) -\rho
}\varepsilon ^{2}\right)  \label{inegalite}
\end{equation}%
where
\begin{equation*}
0<\rho <2a\left( 1-c\right) ,\ K_{1}\leq \exp \left( 2\left( N^{2}+\left(
aS_{1}\max_{i}\mathbb{E}\left\Vert \xi _{i}\right\Vert \right) ^{2}\right)
\right) \ \text{and }K_{2}=\min \left( 1,\frac{1}{16S_{2}}\right) .
\end{equation*}
\end{theorem}

\begin{proof}
Using basic properties of probability and formula (\ref{formule}), we have%
\begin{eqnarray}
\mathbb{P}\left\{ \left\Vert x_{n+1}-x^{\ast }\right\Vert >\varepsilon
\right\} &\leq &\mathbb{P}\left\{ N\prod\limits_{i=1}^{n}\left( 1-\frac{%
a\left( 1-c\right) }{i}\right) +\sum\limits_{i=1}^{n}\frac{a}{i^{2}}%
\prod\limits_{j=i+1}^{n}\left( 1-\frac{a\left( 1-c\right) }{j}\right)
\left\Vert \xi _{i}\right\Vert >\varepsilon \right\}  \notag \\
&\leq &\mathbb{P}\left\{ N\prod\limits_{i=1}^{n}\left( 1-\frac{a\left(
1-c\right) }{i}\right) +\sum\limits_{i=1}^{n}\frac{a}{i^{2}}%
\prod\limits_{j=i+1}^{n}\left( 1-\frac{a\left( 1-c\right) }{j}\right)
\mathbb{E}\left\Vert \xi _{i}\right\Vert \geq \frac{\varepsilon }{2}\right\}
\notag \\
&&+\ \mathbb{P}\left\{ \sum\limits_{i=1}^{n}\frac{a}{i^{2}}%
\prod\limits_{j=i+1}^{n}\left( 1-\frac{a\left( 1-c\right) }{j}\right)
\left( \left\Vert \xi _{i}\right\Vert -\mathbb{E}\left\Vert \xi
_{i}\right\Vert \right) >\frac{\varepsilon }{2}\right\} .  \label{sommation}
\end{eqnarray}%
Let us define
\begin{equation*}
\zeta _{i}=\left\Vert \xi _{i}\right\Vert -\mathbb{E}\left\Vert \xi
_{i}\right\Vert .
\end{equation*}%
It is clear that $\mathbb{E}\zeta _{i}=0$ and $\mathbb{E}\left\vert \zeta
_{i}\right\vert ^{m}\leq 2m!\sigma ^{2}\left( 2L\right) ^{m-2}.$

Firstly, we have%
\begin{equation}
\mathbb{P}\left\{ N\prod\limits_{i=1}^{n}\left( 1-\frac{a\left( 1-c\right)
}{i}\right) +\sum\limits_{i=1}^{n}\frac{a}{i^{2}}\prod\limits_{j=i+1}^{n}%
\left( 1-\frac{a\left( 1-c\right) }{j}\right) \mathbb{E}\left\Vert \xi
_{i}\right\Vert >\frac{\varepsilon }{2}\right\} \leq K_{1}e^{-n^{2a\left(
1-c\right) -\rho }\varepsilon ^{2}}.  \label{terme1}
\end{equation}%
where
\begin{equation*}
K_{1}\leq \exp \left( 2\left( N^{2}+\left( aS_{1}\max_{i}\mathbb{E}%
\left\Vert \xi _{i}\right\Vert \right) ^{2}\right) \right) .
\end{equation*}%
On the other hand, under Markov inequality, we have for all $t>0,$%
\begin{eqnarray*}
&&\mathbb{P}\left\{ \sum\limits_{i=1}^{n}\frac{a}{i^{2}}\prod%
\limits_{j=i+1}^{n}\left( 1-\frac{a\left( 1-c\right) }{j}\right) \left(
\left\Vert \xi _{i}\right\Vert -\mathbb{E}\left\Vert \xi _{i}\right\Vert
\right) >\frac{\varepsilon }{2}\right\} \\
&=&\mathbb{P}\left\{ \sum\limits_{i=1}^{n}\frac{at\left( n+1\right)
^{a\left( 1-c\right) }}{i^{2}}\prod\limits_{j=i+1}^{n}\left( 1-\frac{%
a\left( 1-c\right) }{j}\right) \zeta _{i}>\frac{\varepsilon t\left(
n+1\right) ^{a\left( 1-c\right) }}{2}\right\} \\
&\leq &\exp \left( -\frac{t\varepsilon \left( n+1\right) ^{a\left(
1-c\right) }}{2}\right) \mathbb{E}\exp \left( t\sum\limits_{i=1}^{n}\frac{%
a\left( n+1\right) ^{a\left( 1-c\right) }}{i^{2}}\prod\limits_{j=i+1}^{n}%
\left( 1-\frac{a\left( 1-c\right) }{j}\right) \zeta _{i}\right) .
\end{eqnarray*}

The functions $x\longmapsto \left\Vert x\right\Vert $ and $x\longmapsto
e^{x} $ are continuous, and hence are Borel functions. Therefore, the random
variables%
\begin{equation*}
\exp \left( \sum\limits_{i=1}^{n}\frac{at\left( n+1\right) ^{a\left(
1-c\right) }}{i^{2}}\prod\limits_{j=i+1}^{n}\left( 1-\frac{a\left(
1-c\right) }{j}\right) \zeta _{i}\right)
\end{equation*}%
are also independent. And so,%
\begin{eqnarray*}
&&\mathbb{E}\exp \left( \sum\limits_{i=1}^{n}\frac{at\left( n+1\right)
^{a\left( 1-c\right) }}{i^{2}}\prod\limits_{j=i+1}^{n}\left( 1-\frac{%
a\left( 1-c\right) }{j}\right) \zeta _{i}\right) \\
&=&\prod\limits_{i=1}^{n}\mathbb{E}\exp \left( \frac{at\left( n+1\right)
^{a\left( 1-c\right) }}{i^{2}}\prod\limits_{j=i+1}^{n}\left( 1-\frac{%
a\left( 1-c\right) }{j}\right) \zeta _{i}\right) .
\end{eqnarray*}

The expansion of the exponential function around zero, inequality (\ref%
{produit}) as well as Cramer's condition (\ref{Cramer}) give us,%
\begin{eqnarray*}
&&\mathbb{E}\exp \left( \frac{at}{i^{2}}\prod\limits_{j=i+1}^{n}\left( 1-%
\frac{a\left( 1-c\right) }{j}\right) \zeta _{i}\right) \leq
1+\sum_{m=2}^{+\infty }\frac{a^{m}t^{m}\mathbb{E}\left\vert \zeta
_{i}\right\vert ^{m}}{i^{2m}m!}\left( \frac{i+1}{n+1}\right) ^{a\left(
1-c\right) m} \\
&\leq &1+\frac{2a^{2}t^{2}\sigma ^{2}}{i^{4}}\left( \frac{i+1}{n+1}\right)
^{2a\left( 1-c\right) }\sum_{m=2}^{+\infty }\frac{a^{m-2}t^{m-2}\left(
2L\right) ^{m-2}}{i^{2\left( m-2\right) }}\left( \frac{i+1}{n+1}\right)
^{a\left( 1-c\right) \left( m-2\right) }
\end{eqnarray*}

Note that the function $x\longmapsto \frac{\left( x+1\right) ^{a\left(
1-c\right) }}{x^{2}}$ is decreasing and its maximum on the interval $\left[
1,+\infty \right) $ is $2^{a\left( 1-c\right) }.$ Thus, for suitably chosen $%
t$, e.g.
\begin{equation}
t\leq \frac{\left( n+1\right) ^{a\left( 1-c\right) }}{2^{a\left( 1-c\right)
+2}aL}  \label{t}
\end{equation}

and using the following inequality, $1+x\leq e^{x},$ we get
\begin{equation*}
\prod\limits_{i=1}^{n}\mathbb{E}\exp \frac{at}{i^{2}}\prod%
\limits_{j=i+1}^{n}\left( 1-\frac{a\left( 1-c\right) }{j}\zeta _{i}\right)
\leq \exp \left( \sum_{i=1}^{n}\frac{4a^{2}t^{2}\sigma ^{2}}{i^{4}}\left(
\frac{i+1}{n+1}\right) ^{2a\left( 1-c\right) }\right) .
\end{equation*}%
Consequently,%
\begin{eqnarray}
\mathbb{P}\left\{ \sum\limits_{i=1}^{n}\frac{a}{i^{2}}\prod%
\limits_{j=i+1}^{n}\left( 1-\frac{a\left( 1-c\right) }{j}\right) \zeta _{i}>%
\frac{\varepsilon }{2}\right\} &\leq &\exp \left( -\frac{\varepsilon t}{2}%
+\sum_{i=1}^{n}\frac{4a^{2}t^{2}\sigma ^{2}}{i^{4}}\left( \frac{i+1}{n+1}%
\right) ^{2a\left( 1-c\right) }\right)  \notag \\
&\leq &\exp \left( -\frac{\varepsilon t}{2}+\frac{t^{2}S_{2}}{n^{2a\left(
1-c\right) -\rho }}\right) .  \label{tmin}
\end{eqnarray}

The quantity on the right-hand side of (\ref{tmin}) is minimal at
\begin{equation}
t^{\ast }=\frac{\varepsilon n^{2a\left( 1-c\right) -\rho }}{4S_{2}}.
\label{t*}
\end{equation}%
Thus, by substituting $t^{\ast }$ in (\ref{tmin}), we obtain%
\begin{equation}
\mathbb{P}\left\{ \sum\limits_{i=1}^{n}\frac{a}{i^{2}}\prod%
\limits_{j=i+1}^{n}\left( 1-\frac{a\left( 1-c\right) }{j}\right) \zeta _{i}>%
\frac{\varepsilon }{2}\right\} \leq \exp \left( -\frac{\varepsilon
^{2}n^{2a\left( 1-c\right) -\rho }}{16S_{2}}\right) .  \label{terme2}
\end{equation}%
The conclusion of theorem (\ref{Theoreme}) can be obtained from (\ref%
{sommation}), (\ref{terme1}) and (\ref{terme2}) immediately.
\end{proof}

\begin{remark}
The condition (\ref{Cramer}) is known under Cramer's condition and the first
example that pops to our head \ is the bounded random variables and also the
normal random variables.
\end{remark}

\begin{remark}
Notice that both choices of $t$ in (\ref{t}) and (\ref{t*}) are not
contradictory. Indeed,%
\begin{equation*}
\lim_{n\rightarrow +\infty }\frac{\left( n+1\right) ^{a\left( 1-c\right) }}{%
n^{a\left( 1-c\right) -\rho }}=+\infty \Longleftrightarrow \forall A\in
\mathbb{R}^{+},\exists ~n_{0}\in \mathbb{N}:n\geq n_{0}\Longrightarrow \frac{%
\left( n+1\right) ^{a\left( 1-c\right) }}{n^{a\left( 1-c\right) -\rho }}>A.
\end{equation*}

For $A=\frac{2^{a\left( 1-c\right) +2}aLS_{2}\varepsilon }{4S_{2}},$ we have
\begin{equation*}
\frac{\varepsilon n^{a\left( 1-c\right) -\rho }}{4S_{2}}<\frac{\left(
n+1\right) ^{a\left( 1-c\right) }}{2^{a\left( 1-c\right) +2}aL}.
\end{equation*}
\end{remark}

\subsection{Almost complete convergence}

As a direct consequence of theorem (\ref{Theoreme}), we obtain the almost
complete convergence ($a.co$.) of the Mann's stochastic scheme.

\begin{corollary}
Under the assumptions of theorem (\ref{Theoreme}), the algorithm (\ref{smi})
converges almost completely ($a.co.$) to the unique fixed-point $x^{\ast }$
of $F.$
\end{corollary}

\begin{proof}
Indeed, since the series of general term%
\begin{equation}
u_{n}=K_{1}\exp \left( -K_{2}n^{a\left( 1-c\right) -\rho }\varepsilon
^{2}\right)
\end{equation}%
is convergent, we have, for all $\varepsilon >0,$%
\begin{equation}
\sum\limits_{n=1}^{\infty }\mathbb{P}\left\{ \left\Vert x_{n+1}-x^{\ast
}\right\Vert >\varepsilon \right\} <+\infty ,
\end{equation}%
which ensures the almost complete convergence.
\end{proof}

\begin{remark}
Notice that if $\left( x_{n}\right) _{n}$ converges almost completely
towards $x^{\ast }$ then it also converges almost surely to $x^{\ast }$. In
other words, if the sequence $\left( x_{n}\right) _{n}$ converges in
probability to $x^{\ast }$ sufficiently quickly (i.e. the above sequence of
tail probabilities is summable for all $\varepsilon >0$), then the sequence $%
\left( x_{n}\right) _{n}$ also converges almost surely to $x^{\ast }$. This
is a direct implication from the Borel--Cantelli's lemma.
\end{remark}

\subsection{Confidence set}

In this subsection, we build a confidence set for the fixed point of a
contraction mapping determined by the stochastic Mann's algorithm.

\begin{corollary}
Under the assumptions of theorem (\ref{Theoreme}), for a given level $\alpha
$, there is a natural integer $n_{\alpha }$ for which the fixed point $%
x^{\ast }$ of $F$ belongs to the closed ball of center $x_{n_{\alpha }+1}$
and radius $\varepsilon $ with a probability greater than or equal to $%
1-\alpha $. In other words,
\begin{equation}
\forall \ \varepsilon >0,\forall \ \alpha >0,\exists \ n_{\alpha }\in
\mathbb{N}:\mathbb{P}\left\{ \left\Vert x_{n_{\alpha }+1}-x^{\ast
}\right\Vert \leq \varepsilon \right\} \geq 1-\alpha .  \label{ic}
\end{equation}
\end{corollary}

\begin{proof}
We have,
\begin{equation}
\lim_{n\rightarrow +\infty }K_{1}\exp \left( -K_{2}n^{a\left( 1-c\right)
-\rho }\varepsilon ^{2}\right) =0.
\end{equation}%
Since there exists a natural integer $n_{\alpha }$ such that%
\begin{equation}
\forall \ n\in \mathbb{N},n\geq n_{\alpha }\Longrightarrow K_{1}\exp \left(
-K_{2}n^{a\left( 1-c\right) -\rho }\varepsilon ^{2}\right) \leq \alpha ,
\label{fin}
\end{equation}%
then, (\ref{ic}) arises from (\ref{inegalite}) and (\ref{fin}).
\end{proof}

\begin{remark}
Conversely, if the sample size is given, we can also determine by (\ref%
{inegalite}) the level of significance $\alpha $ required in the
construction of the confidence set.
\end{remark}

\subsection{Rate of convergence}

In this subsection, we study the rate of convergence of \ the Mann's
stochastic algorithm (\ref{mannstocha}). We say that $x_{n}-x^{\ast
}=O\left( r_{n}\right) ,$ almost completely ($a.co.$) where $\left(
r_{n}\right) _{n}$ is a sequence of real positive numbers, if there exists $%
\epsilon _{0}>0,\epsilon _{0}=O\left( 1\right) $ such that

\begin{equation*}
\sum_{n=1}^{+\infty }\mathbb{P}\left\{ \left\Vert x_{n}-x^{\ast }\right\Vert
>\epsilon _{0}r_{n}\right\} <+\infty .
\end{equation*}

\begin{theorem}
For all $a>0$ satisfying $a\left( 1-c\right) <1,$ we have%
\begin{equation}
x_{n+1}-x^{\ast }=O(\sqrt{\frac{\ln n}{n^{a\left( 1-c\right) -\rho }}}),\ \
\ \ a.co.  \label{vitesse}
\end{equation}
\end{theorem}

\begin{proof}
Indeed, we have%
\begin{equation}
\mathbb{P}\left\{ \left\Vert x_{n+1}-x^{\ast }\right\Vert >\varepsilon
\right\} \leq K_{1}e^{-K_{2}n^{a\left( 1-c\right) -\rho }\varepsilon ^{2}},
\label{inegeneral}
\end{equation}%
where $K_{1}$ and $K_{2}$ are positive constants.

Consequently,%
\begin{equation}
\mathbb{P}\left\{ \left\Vert x_{n+1}-x^{\ast }\right\Vert >\epsilon _{0}%
\sqrt{\frac{\ln n}{n^{a\left( 1-c\right) -\rho }}}\right\} \leq
K_{1}n^{-k_{2}\epsilon _{0}^{2}}.
\end{equation}%
For $\epsilon _{0}$ well chosen, for example $\epsilon _{0}=\sqrt{\frac{1+d}{%
K_{2}}},d>0$, the right and-side of the inequality (\ref{inegeneral}) is a
general term of convergent series. Hence, the desired result (\ref{vitesse})
is proved.
\end{proof}

\section{Numerical illustrations}

In order to ask the feasibility of the presented algorithm and check the
obtained results of convergence, we consider a numerical example where we
take a known contraction function $F$ and thus possessing a unique fixed
point. By using the Mann's algorithm, we obtain the approximated fixed point
and we compare it with the exact one by giving the absolute and relative
error. Concerning the independent random errors $\left( \xi _{n}\right) _{n}$
introduced in the algorithm, we take them following a reduced centred normal
distribution.

\bigskip

Consider $\mathbb{B}=\mathbb{R}$ and the following function $F$ defined by
\begin{equation*}
\begin{array}{ll}
F: & \mathbb{R}\rightarrow \mathbb{R} \\
& x\mapsto F\left( x\right) =\frac{1}{1+x^{2}}%
\end{array}%
\end{equation*}

It is clear that $F$ is a contraction, moreover, we have
\begin{equation*}
\left\vert F\left( x\right) -F\left( y\right) \right\vert \leq \frac{9}{8%
\sqrt{3}}\left\vert x-y\right\vert <0.65\left\vert x-y\right\vert
\end{equation*}

Consequently, the function $F$ has a unique fixed point given by:%
\begin{equation*}
\sqrt[3]{\sqrt{\frac{31}{108}}+\frac{1}{2}}-\frac{1}{3\sqrt[3]{\sqrt{\frac{31%
}{108}}+\frac{1}{2}}}\simeq 0.682327803828019
\end{equation*}

\begin{enumerate}
\item \textbf{The approximated values of the fixed point}\newline
\newline
\qquad Here, we give the approximated values of the fixed point for
different number of iterations $n$. To compare the fixed point to the
approximated ones, we give the absolute error and relative one. For an
arbitrary choice of $x_{1}$, namely $x_{1}=0.5,\ $the obtained numerical
results are represented in the following table.\newline
\newline
\begin{equation*}
\begin{tabular}{|l|l|l|l|}
\hline
$n$ & $x_{n}$ & Absolute error & Relative error \\ \hline
$10$ & $0.751990575294321$ & $0.069662771466302$ & $0.102095753793818$ \\
\hline
$100$ & $0.689797221706968$ & $0.007469417878949$ & $0.010946963962254$ \\
\hline
$1000$ & $0.683175136988214$ & $8.473331601946965\ e-004$ & $1.241827103397\
e-003$ \\ \hline
$10^{4}$ & $0.682471996264786$ & $1.441924367667768\ e-004$ & $%
2.113242871795981\ e-004$ \\ \hline
$10^{5}$ & $0.682336379893621$ & $8.576065601673122\ e-006$ & $%
1.256883502850006\ e-005$ \\ \hline
$10^{6}$ & $0.682328662137050$ & $8.583090307379138\ e-007$ & $%
1.257913023509519\ e-006$ \\ \hline
$10^{7}$ & $0.682327889660202$ & $8.583218269464510\ e-008$ & $%
1.257931777264628\ e-007$ \\ \hline
\end{tabular}%
\end{equation*}%
\newline
\newline
Note that from $n=1000$, the approximated fixed point is very close to the
real one. These results show the efficiency of the Mann's iterative scheme,
also this method is very easy to implement under the programming package
Matlab. \newline

\item \textbf{level of significance }$\mathbf{\alpha }$
\end{enumerate}

In the following two tables, we take a level of significance $\alpha $ and
for different $\varepsilon ,$ we give the order of number of iterations and
hence after implementing the algorithm we obtain the corresponding
approximated fixed point and consequently a confidence interval.

\begin{description}
\item[(i)] $\alpha =0.01$%
\begin{equation*}
\begin{tabular}{|c|c|c|}
\hline
$\varepsilon $ & $n$ & confidence interval \\ \hline
$0.01$ & $10^{7}$ & $\left[ 0.672327889660202,0.692327803828019\right] $ \\
\hline
$0.05$ & $10^{5}$ & $\left[ 0.632336379893621,0.782336379893621\right] $ \\
\hline
$0.1$ & $10^{4}$ & $\left[ 0.582471996264786,0.782471996264786\right] $ \\
\hline
\end{tabular}%
\end{equation*}

\item[(ii)] $\alpha =0.05$%
\begin{equation*}
\begin{tabular}{|c|c|c|}
\hline
$\varepsilon $ & $n$ & Confidence interval \\ \hline
$0.01$ & $10^{6}$ & $\left[ 0.672328662137050,0.692328662137050\right] $ \\
\hline
$0.05$ & $10^{4}$ & $\left[ 0.632471996264786,0.732471996264786\right] $ \\
\hline
$0.1$ & $10^{3}$ & $\left[ 0.583175136988214,0.783175136988214\right] $ \\
\hline
\end{tabular}%
\end{equation*}
\end{description}


\begin{thebibliography}{99}
\bibitem{Abe} R.T.\ Abebe and H. Zegeye, Mann and Ishikawa-Type Iterative
Schemes for Approximating Fixed Points of Multi-valued Non-Self Mappings,
Mediterranean Journal of Mathematics, $\left( 2016\right) ,1-16$.

\bibitem{Aga} R.P\ Agarwal, N.J\ Huang and Y.J\ Chao, Stability of iterative
process with errors for nonlinear equations of $\phi $-strongly accretive
type operators, Numer. Funct. Anal. Optimiz., $22(5$ \& $6)$, $\left(
2001\right) ,471-485$.

\bibitem{Agr} N.K.\ Agrawa and G. Gupta, Two step random iteration scheme
for random contractive operators in uniformly separable convex Banach space,
International Journal of Recent Advances in Multidisciplinary Research, Vol.
03, Issue 02, $\left( 2016\right) ,$ $1197-1206$.

\bibitem{Ar} A. Arunchai and S. Plubtieng, Random fixed point theorem of
Krasnoselskii type for the sum of two operators, Fixed Point Theory and
Applications $\left( 2013\right) $, 2013:142

\bibitem{Beg} I. Beg, D.\ Dey and M.\ Saha, Convergence and stability of two
random iteration algorithms, J. Nonlinear Funct. Anal. $\left( 2014\right) $%
, 2014:17

\bibitem{Ber} V. Berinde, Iterative Approximation of Fixed Points, Lecture
Notes in Mathematics $1912$, Springer, $2007.$

\bibitem{Ceg} A. Cegielski, Iterative Methods for Fixed Point Problems in
Hilbert Spaces, Springer, $2012.$

\bibitem{Cha} S.S. Chang, Y.J. Cho, B.S. Lee, J.S. Jung and S. M. Kang,
Iterative Approximations of Fixed Points and Solutions for Strongly
Accretive and Strongly Pseudo-Contractive Mappings in Banach Spaces, Journal
of Mathematical Analysis an Applications, $224$, $\left( 1998\right)
,149-165.$

\bibitem{Cha2} S.S. Chang and J.K\ Kim. Convergence Theorems of the Ishikawa
Type Iterative Sequences with Errors for Generalized Quasi-Contractive
Mappings in Convex Metric Spaces, Applied Mathematics Letters, $%
16,(2003),535-542.$

\bibitem{Cho0} Y. J. Cho, J. Li and N. J. Huang, Random Ishikawa iterative
sequence with errors for approximating random fixed points. Taiwanese
Journal of Mathematics, Vol.12, N$%
{{}^\circ}%
$1, $\left( 2008\right) ,51-61$.

\bibitem{Cho} B.S. Choudhury, Random Mann iteration sequence, Applied Math.
Lett., $16$ ($2003$),$93-96$.

\bibitem{Cho1} B. S. Choudhury and M. Ray, Convergence of an iteration
leading to a solution of a random operator equation, J. Appl. Math. Stoc.
Anal., $12$ ($1999$),$161-168$.

\bibitem{Chu0} R. Chugh, R. Rani and S. Narwal, Common Fixed Theorems Using
Random Implicit Iterative Schemes, International Journal of Engineering And
Science, Vol.4, Issue 8 ($2014$),$61-69.$

\bibitem{Chu} R. Chugh, V. Kumar, and S Narwalwe, Some strong convergence
results of random iterative algorithms with errors in Banach spaces, Commun.
Korean Math. Soc. 31 ($2016$), N$%
{{}^\circ}%
$1, $147-161.$

\bibitem{Hua} Z. Huang, Mann and Ishikawa Iterations with Errors for
Asymptotically Nonexpansive Mappings, Computers and Mathematics with
Applications, $37$, $(1999),1-7.$

\bibitem{Hus} N.Hussaina, S. Narwalb, R. Chughc and V. Kumard, On
convergence of random iterative schemes with errors for strongly
pseudo-contractive Lipschitzian maps in real Banach spaces, J. Nonlinear
Sci. Appl. 9 ($2016$),$3157-3168.$

\bibitem{Ish} S. Ishikawa, Fixed points by a new iteration method, Proc.
Amer. Math. Soc., Vol. $44$, N$%
{{}^\circ}%
1,$ $\left( 1974\right) ,147-150.$

\bibitem{Kan} S.M. Kang, F.\ Ali, R.\ Arif, Y. C. Kwun and S. Jabeen, On the
Convergence of Mann and Ishikawa Type Iterations in the Class of Quasi
Contractive Operators, Journal of Computational Analysis \& Applications.
Vol. 21 Issue 1, $\left( 2016\right) ,451-459$.

\bibitem{Kaz} K.R. Kazmi, Mann and Ishikawa Type Perturbed Iterative
Algorithms for Generalized Quasivariational Inclusions, Journal of
Mathematical Analysis an Applications. $209$, $\left( 1997\right) ,$ $%
572-584.$

\bibitem{Kim} G.E.\ Kim and T.H.\ Kim. Mann and Ishikawa Iterations with
Errors for Non-Lipschitzian Mappings in Banach Spaces, Computers and
Mathematics with Applications $42$, $(2001),1565-1570.$

\bibitem{Kra} M.A. Krasnosel'skii, Two remarks on the method of successive
approximations, Uspekhi Mat. Nauk, Volume 10, Issue 1(63), $\left(
1955\right) ,123-127.$

\bibitem{Liu} L.S\ Liu, Ishikawa and Mann iterative process with errors for
nonlinear strongly accretive mappings in Banach spaces, Journal of
Mathematical Analysis an Applications, $194,$ $\left( 1995\right) ,114-125.$

\bibitem{Liu2} L.S.\ Liu, Fixed points of local strictly pseudo-contractive
mappings using Mann and Ishikawa iteration with errors, Indian J. Pure Appl.
Math., $26(7),\left( 1995\right) ,649-659.$

\bibitem{Liu3} Z.\ Liu and S.M\ Kang. Stability of Ishikawa Iteration
Methods with Errors for Strong Pseudocontractions and Nonlinear Equations
Involving Accretive Operators in Arbitrary Real Banach Spaces, Mathematical
and Computer Modelling, $34$, ($2001$), $319-330$.

\bibitem{ZLiu2} Z.\ Liu, J.K. Kim and J.S Ume, Stability of Ishikawa
iteration schemes with errors for nonlinear accretive operators in arbitrary
Banach spaces, Nonlinear Funct. Anal. Appl., 7$\left( 1\right) ,$ $\left(
2002\right) ,55-67$.

\bibitem{Man} W.R. Mann, Mean value methods in iteration, Proc. Amer. Math.
Soc., 4, $(1953),506-510$.

\bibitem{Oke} G.A. Okeke and J.K. Kim, Convergence and summable almost
T-stability of the random Picard-Mann hybrid iterative process, Journal of
Inequalities and Applications ($2015$) 2015:290

\bibitem{Osi} M.O. Osilike, Ishikawa and Mann Iteration Methods with Errors
for Nonlinear Equations of the Accretive Type, Journal of Mathematical
Analysis an Applications, $213$, $\left( 1997\right) ,91-105.$

\bibitem{Osi2} M.O. Osilike, Stability of the Mann and Ishikawa Iteration
Procedures for f-Strong Pseudocontractions and Nonlinear Equations of the $%
\phi $-Strongly Accretive Type, Journal of Mathematical Analysis an
Applications, $227$, $\left( 1998\right) ,319-334.$

\bibitem{Pan} B \ Panyanak, Mann and Ishikawa iterative processes for
multivalued mappings in Banach spaces, Computers and Mathematics with
Applications, $54$, $(2007),872-877.$

\bibitem{Pic} E. Picard, M\'{e}moire sur la th\'{e}orie des \'{e}quations
aux d\'{e}riv\'{e}es partielles et la m\'{e}thode des approximations
successives, Journal de Math\'{e}matiques pures et appliqu\'{e}es, 6, $%
\left( 1890\right) ,145-210$.

\bibitem{Sal} S. Saluja, D. Magarde and A.K.\ Dhakde, A Common Fixed Point
Theorem for Two Random Operators using Random Mann Iteration Scheme,
Mathematical Theory and Modeling, Vol.3, No.6, $\left( 2013\right) ,263-266.$

\bibitem{Sha} N. Shahzada and H. Zegeyeb, On Mann and Ishikawa iteration
schemes for multi-valued maps in Banach spaces. Nonlinear Analysis: Theory,
Methods \& Applications, Volume 71, Issues 3--4, $\left( 2009\right)
,838-844 $.

\bibitem{HXu} H-K Xu, A Note on the lshikawa iteration Scheme. Journal of
Mathematical Analysis an Applications, 167, ($1992$),$582-587$.

\bibitem{YXu} Y.\ Xu, Ishikawa and Mann Iterative Processes with Errors for
Nonlinear Strongly Accretive Operator Equations. Journal of Mathematical
Analysis an Applications. $224$, $\left( 1998\right) ,91-101$.

\bibitem{YXu2} Y. Xu, Ishikawa and Mann iterative methods with errors for
nonlinear accretive operator equations, J.\ Math. Anal. Appl. 224, $\left(
1998\right) ,91-101$.
\end{thebibliography}
\end{document}